\documentclass[11pt]{amsart}
\usepackage[english]{babel}
\usepackage{inputenc}

\usepackage{latexsym}
\usepackage{amssymb}
\usepackage{amsmath}
\usepackage{amsfonts}
\usepackage[mathscr]{eucal}
\usepackage{xcolor}

\usepackage{amscd}

\usepackage{graphicx}
\usepackage{epsfig}
\usepackage{relsize}

\usepackage{fullpage}

\usepackage{longtable}
\usepackage{rotating}
\usepackage{array}
\usepackage{multicol}
\usepackage{multirow}
\usepackage{makecell}
\usepackage{mathtools}
\usepackage{booktabs}

\usepackage{amsthm}
\usepackage[all,cmtip]{xy}
\usepackage[pdfstartview=FitH]{hyperref}
\usepackage{comment}
\usepackage{psfrag}
\usepackage{verbatim}
\usepackage{hyperref}
\hypersetup{
colorlinks=true,
linkcolor=blue
}
\usepackage{pifont}
\usepackage{caption}
\usepackage{csquotes}

\usepackage{soul}
 \usepackage[normalem]{ulem}

\usepackage{tikz}
\usepackage{pgfplots}
\pgfplotsset{compat=1.15}
\usepackage{mathrsfs}
\usetikzlibrary{arrows}

\DeclareMathOperator{\identity}{Id}

\usepackage[style=numeric,backend=biber,maxbibnames=99]{biblatex}

\newtheorem{thm}{Theorem}[section]
\newtheorem*{thm*}{Theorem}
\newtheorem{cor}[thm]{Corollary}	
\newtheorem*{cor*}{Corollary}	
\newtheorem{coro}{Corollary}

\newtheorem{prop}[thm]{Proposition}

\newtheorem{theorem}{Theorem}
\renewcommand*{\thetheorem}{\Alph{theorem}}

\theoremstyle{definition}
\newtheorem{definition}[thm]{Definition} 	
\newtheorem{rema}[thm]{Remark} 
\newtheorem*{definition*}{Definition}    	

\numberwithin{equation}{section}

\DeclareMathOperator{\dist}{dist}

\DeclareMathOperator{\diam}{diam}
\DeclareMathOperator{\proj}{proj}

\DeclareMathOperator{\reach}{reach}
\DeclareMathOperator{\unp}{Unp}

\DeclareMathOperator{\dgm}{Dgm}
\DeclareMathOperator{\cost}{cost}

\newcommand{\norm}[1]{\left\Vert#1\right\Vert}

\newcommand{\RR}{\mathbb{R}}

\newcommand{\EE}{\mathbb{E}}

\begin{document}



\title[W&R]{On the reach of isometric embeddings into Wasserstein type spaces}


\author[J.~Casado]{Javier Casado$^{\ast}$}

\author[M.~Cuerno]{Manuel Cuerno$^{\ast\ast}$}

\author[J.~Santos-Rodríguez]{Jaime Santos-Rodríguez$^{\ast\ast\ast}$}


\thanks{$^*$Supported in part by the FPU Graduate Research Grant FPU20/01444, and by research grants  
	 MTM2017-‐85934-‐C3-‐2-‐P and PID2021-124195NB-C32
from the Ministerio de Econom\'ia y Competitividad de Espa\~{na} (MINECO)} 

\thanks{$^{\ast\ast}$Supported in part by the FPI Graduate Research Grant PRE2018-084109, and by research grants  
	 MTM2017-‐85934-‐C3-‐2-‐P and PID2021-124195NB-C32
from the Ministerio de Econom\'ia y Competitividad de Espa\~{na} (MINECO)}
\thanks{$^{\ast\ast\ast}$ Supported in part by a Margarita Salas Fellowship CA1/RSUE/2021-00625, and by research grants  
	 MTM2017-‐85934-‐C3-‐2-‐P, PID2021-124195NB-C32
from the Ministerio de Econom\'ia y Competitividad de Espa\~{na} (MINECO)}


\address[J.~Casado]{Department of Mathematics, Universidad Aut\'onoma de Madrid and ICMAT CSIC-UAM-UC3M, Spain}
\email{javier.casadoa@uam.es} 

\address[M.~Cuerno]{Department of Mathematics, Universidad Aut\'onoma de Madrid and CSIC-UAM-UC3M, Spain and Department of Mathematical Sciences, Durham University, UK}
\email{manuel.mellado@uam.es, manuel.mellado-cuerno@durham.ac.uk}

\address[J.~Santos-Rodr\'iguez]{Department of Mathematics, Universidad Aut\'onoma de Madrid, Spain and Department of Mathematical Sciences, Durham University, UK}
\email{jaime.santos@uam.es, jaime.santos-rodriguez@durham.ac.uk}


\date{\today}


\subjclass[2020]{49Q20, 28A33, 30L15, 49Q22, 53C21, 55N31}
\keywords{Wasserstein distance, Optimal transport, Metric geometry, TDA}


\begin{abstract}
We study the reach (in the sense of Federer) of the natural isometric embedding $X\hookrightarrow W_p(X)$ of $X$ inside its $p$-Wasserstein space, where $(X,\dist)$ is a geodesic metric space. We prove that if a point $x\in X$ can be joined to another point $y\in X$ by two minimizing geodesics, then $\reach(x, X\subset W_p(X)) = 0$. This includes the cases where $X$ is a compact manifold or a non-simply connected one. On the other hand, we show that $\reach(X\subset W_p(X)) = \infty$ when $X$ is a CAT(0) space. The infinite reach enables us to examine the regularity of the projection map. Furthermore, we replicate these findings by considering the isometric embedding $X\hookrightarrow W_\vartheta(X)$ into an Orlicz--Wasserstein space, a generalization by Sturm of the classical Wasserstein space. Lastly, we establish the nullity of the reach for the isometric embedding of $X$ into $\dgm_\infty$, the space of persistence diagrams equipped with the bottleneck distance.
\end{abstract}
\setcounter{tocdepth}{1}

\maketitle

\section{Introduction}
The concept of the reach of a subset in Euclidean space was first introduced by Federer in \cite{federer}.  It is  used as a way to measure how much the subset folds in on itself (i.e. how far apart two pieces of the set are in the ambient  space despite them being far in the intrinsic metric of the set).
Loosely speaking (see Definition \ref{def.reach})  a subset $A\subset X$ has positive reach if there is a neighbourhood of any point on $A$ such that every point in this neighbourhood has a unique metric projection into $A$. That is, every $x\in X$ inside that neighbourhood is sent by the projection to its unique nearest point in $A$.

The reach of a subset has been of interest not only for its geometric and topological properties (see for example \cite{lytchakkapovitchreach,lytchakreach2,lytchakreach1}) but also for its application as a useful parameter for manifold learning and topological data analysis (see \cite{harvey,latschevreach} and references therein). In the survey \cite{surveyreach}, the interested reader can also find a summary of some results for sets of positive reach. 

Given a geodesic metric space $(X, \dist),$ one can equip the space of probability measures supported on $X$ with a distance induced by the solutions to  an optimal transport problem.  Usually the cost comes from taking the $p-$power of the distance function, the so called $p-$\textit{Wasserstein spaces.} 
One advantage of considering these ambient spaces is that they  share many geometrical properties with the base space $X$ such as non-branching of geodesics, compactness, and  lower sectional curvature bounds amongst others. 

In this article we focus on determining the reach of the image of the natural isometric embedding, given by mapping each point $x \in X$ to the corresponding Dirac delta $\delta_x\in W_p(X).$ We denote this 
by \(\reach(X\subset W_p(X))\), where $W_p(X)$ is the $p$--Wasserstein space of $X$ (see Section \ref{wassersteintype}).

Our first result shows that the cost considered affects the reach of the embedding significantly:
\begingroup
\def\thetheorem{\ref{reach1wasserstein}}
\begin{theorem}
Let $(X,\dist)$ be a metric space, and consider its $1$--Wasserstein space, $W_1(X).$ Then, for every accumulation point $x\in X$, $\reach(x, X\subset W_1(X)) = 0$. In particular, if $X$ is not discrete, $\reach(X\subset W_1(X)) = 0$.
\end{theorem}
\addtocounter{theorem}{-1}
\endgroup

Geometric features of $X$ also play an important role. In the presence of multiple geodesics joining the same pair of points we obtain the following:


\begingroup
\def\thetheorem{\ref{teorema22}}
\begin{theorem}
 Let $X$ be a geodesic metric space, and $x\in X$ a point such that there exists another $y\in X$ with the property that there exist at least two different minimising geodesics from $x$ to $y$. Then, for every $p>1,$ \[
\reach(x, X\subset W_p(X))=0.
\]    
In particular, if there exists a point $x\in X$ satisfying that property, $\reach(X\subset W_p(X))=0$ for every $p>1$.
\end{theorem}
\addtocounter{theorem}{-1}
\endgroup

This theorem leads us to obtain two interesting corollaries related to two important classes of manifolds:

\begingroup
\def\thecoro{\ref{corollarycompact}}
\begin{coro}
If $M$ is a compact manifold, then $\reach(x,M\subset W_p(M)) =0$ for every $x\in M$ and every $p>1$. 
\end{coro}
\addtocounter{coro}{-1}
\endgroup

\begingroup
\def\thecoro{\ref{corollarynotsimply}}
\begin{coro}
If $M$ is a not simply connected complete manifold, then $\reach(x,M\subset W_p(M)) =0$ for every $x\in M$ and $p>1$.
\end{coro}
\addtocounter{coro}{-1}
\endgroup

In \cite{kell}, Kell studied several convexity conditions, such as \textit{(resp. strictly, uniformly) $p$--convexity} or  \textit{Busemann},  on the distance of a geodesic metric space and some other more general conditions about metric spaces, such as \textit{reflexivity} (see also Definitions \ref{kelldefinitions1} and \ref{kelldefinitions2}) obtaining existence and uniqueness of \textit{barycenters}, i.e., certain points on the metric space that minimise the distance to a given measure/density. In order to formalise that concept, we define a \emph{barycenter} as a point in $X$ that minimises the distance between a given element from a Wasserstein type space and some isometric embedding of the metric space inside that space. Kell's conditions allow us to determine that the reach is infinite for a broad class of spaces:


\begingroup
\def\thetheorem{\ref{reachpositivowass}}
\begin{theorem}
Let $(X,\dist)$ be a reflexive metric space, Then
the following assertions hold:
    \begin{enumerate}
        \item If $X$ is strictly $p$--convex for $p\in[1,\infty)$ and uniformly $\infty$--convex if $p=\infty$, then\begin{equation}
            \reach(X\subset W_r(X))=\infty\text{ for }r>1.
        \end{equation} 
        \item If $X$ is Busemann, strictly $p$--convex for some $p\in[1,\infty],$ and uniformly $q$--convex for some $q\in[1,\infty]$, then\begin{equation}
            \reach(X\subset W_r(X))=\infty\text{ for }r>1.
        \end{equation}
    \end{enumerate}
\end{theorem}
\addtocounter{theorem}{-1}
\endgroup

We also study properties of the projection map
\textit{projection map}, i.e., \begin{align*}
    \proj_2: W_2(X)&\to X\\
    \mu&\mapsto r_\mu,
\end{align*}that sends each measure to its $2-$barycenter (i.e. the barycenter on the $2$--Wasserstein space), and showing that this map is in fact a submetry for a certain class of spaces.

\begingroup
\def\thetheorem{\ref{thmsubmetry}}
\begin{theorem}
   Let $(\EE^n,\dist)$ the Euclidean space with the canonical distance, then $\proj_2$ is a submetry. 
\end{theorem}
\addtocounter{theorem}{-1}
\endgroup

Our next results focus on  the embedding into Wasserstein type spaces with more general metrics, such as the Orlicz-Wasserstein spaces defined by Sturm in \cite{sturm} and the space of persistence diagrams, the key tool in Topological Data Analysis \cite{chazalintro}. 

For the Orlicz-Wasserstein spaces we require reasonable assumptions on the cost (stated in the theorem) in order to ensure that the natural embedding using Dirac deltas is indeed isometric. In a similar fashion to the case of $p-$Wasserstein spaces, we obtain:

\begingroup
\def\thetheorem{\ref{reachceroorlicz}}
\begin{theorem}
Let $X$ be a geodesic metric space, and $x\in X$ a point such that there exists another $y\in X$ with the property that there exists at least two different minimising geodesics from $x$ to $y$. Suppose $X$ is isometrically embedded into an Orlicz-Wasserstein space $W_\vartheta(X)$. Then, for every $\varphi$ (as explained in Subsection \ref{subseccionorlicz}) such that $\varphi(t_0) \neq t_0$ for some $t_0>1$, \[
\reach(x, X\subset W_\vartheta(X))=0.
\]    
In particular, if there exists a point $x\in X$ satisfying that property, $\reach(X\subset W_\vartheta(X))=0$ for every $p>1$.
\end{theorem}
\addtocounter{theorem}{-1}
\endgroup


The last case of isometric embedding into a Wasserstein type space is the one into $\dgm$, the space of persistence diagrams. We can equip $\dgm$ with Wasserstein type distances, involving a minimisation process as in an optimal transport problem. The \textit{bottleneck distance}, $w_\infty$, is one of the most used distances in $\dgm$. Bubenik and Wagner proved in \cite{bubenik} the existence of an isometric embedding of separable and bounded metric spaces into $(\dgm_\infty,w_\infty)$. We have studied the reach of these embeddings:

\begingroup
\def\thetheorem{\ref{reachPD}}
\begin{theorem}
Let $(X,\dist)$ be a separable, bounded metric space and $(\dgm_{\infty},w_\infty)$ the space of persistence diagrams with the bottleneck distance. If $x\in X$ is an accumulation point, then \[
    \reach(x, X\subset \dgm_\infty)=0.\] 
    In particular, if $X$ is not discrete, $\reach(X\subset \dgm_\infty)=0.$
\end{theorem}
\addtocounter{theorem}{-1}
\endgroup

The paper is organized as follows: In Section \ref{metricdefinitions}, we state the necessary technical definitions that we need. Section \ref{wassersteintype} is devoted to present the Wasserstein type spaces we use and some of their properties. Sections \ref{reachwasserstein}, \ref{reachwassersteinorliz} and \ref{reachpersistence} contain the results about the reach of spaces embedded into their $p-$Wasserstein space, their Orlicz--Wassertein space and the persistence diagram space respectively.

The authors would like to express their sincere gratitude to Professor Luis Guijarro for his invaluable comments and insights during the elaboration of this paper. They would also like to extend their appreciation to Professors Fernando Galaz-García and David González for their enlightening discussions and contributions to the final manuscript.

Additionally, the authors wish to acknowledge the Department of Mathematical Sciences at Durham University for their warm hospitality and the excellent working conditions provided during the final months of preparing this paper.

\section{Preliminaries}{\label{metricdefinitions}}
\subsection{Reach}
First we recall the definition of the reach of a subset of a metric space. 
\begin{definition}[Unique points set and reach, \cite{federer}]\label{def.reach}
    Let $(X, \dist)$ be a metric space and $A\subset X$ a subset. We define the set of points having a unique metric projection in $A$ as
    \[\unp(A) = \{x \in X : \text{there exists a unique $a$ such that }  \dist(x,A) = \dist(x,a)\}.\]
    For $a\in A$, we define the \emph{reach} of $A$ at $a$, denoted by $\reach(a, A)$, as
    \[\reach(a, A) = \sup \{ r\ge 0 : B_r(a) \subset \unp(A)\}.\]
    Finally, we define the \emph{global reach} by 
    \[\reach(A) = \inf_{a\in A} \reach(a, A).\]
\end{definition}
The intuitive idea is that $\reach(A)=0$ if and only if we do \textit{not} have an $\varepsilon$--neighbourhood of $A$ admitting a metric projection into $A$. Conversely, $\reach(A)=\infty$ will occur if and only if the entirety of $X$ admits a metric projection into $A$.

\subsection{CAT(0) spaces} We recall the definition of a CAT(0) metric space.

\begin{definition}
    A complete metric space $(X,\dist)$ is CAT(0) if for all $z$, $y \in X$ there exists $m\in X$ such that for all  $x\in X$,
    \[\dist(x,m)^2\leq \cfrac{\dist(x,y)^2 + \dist(x,z)^2 }2 - \cfrac{\dist(y,z)^2}4.\]
\end{definition}
This is a generalization of the concept of nonpositive curvature for Riemannian manifolds to metric spaces. So, in particular, the Euclidean space or the hyperbolic space are examples of CAT(0) spaces. A few basic properties of these spaces are:
\begin{enumerate}
    \item For any two points in $X$, there exists a unique geodesic segment between them.
    \item $X$ is simply connected.
\end{enumerate}

\subsection{General metric definitions}

Most of the definitions in this section appear in \cite[Section 1]{kell}. First, we recall the well known definition of existence of \textit{midpoints}:

\begin{definition}[Midpoints]
    We say that $(X,\dist)$ admits \textit{midpoints} if , for every $x,y\in X$, there is $m(x,y)\in X$ such that \[\dist(x,m(x,y))=\dist(y,m(x,y))=\frac12\dist(x,y).\]
\end{definition}

This technical detail allows us to present the following definitions:

\begin{definition}[$p$--convex, $p$--Busemann curvature  and uniformly $p$--convex]{\label{kelldefinitions1}}
Let $(X,\dist)$ be a metric space that admits midpoints.
\begin{enumerate}
    \item $X$ is \textit{$p$--convex} for some $p\in[1,\infty]$ if, for each triple $x,y,z\in X$ and each midpoint $m(x,y)$ of $x$ and $y$, \[
    \dist(m(x,y),z)\leq\left(\frac12\dist(x,z)^p+\frac12\dist(y,z)^p\right)^{1/p}.
    \]The space $X$ is called \textit{strictly $p$--convex} for $p\in(1,\infty]$ if the inequality is strict for $x\neq y$ and \textit{strictly $1$--convex} if the inequality is strict whenever $\dist(x,y)>|\dist(x,z)-\dist(y,z)|$.
    \item $X$ satisfies the \textit{$p$--Busemann curvature condition} if, for all $x_0,x_1,y_0,y_1\in X$ with midpoints $m_x=m(x_0,x_1)$ and $m_y=m(y_0,y_1)$, \[
    \dist(m_x,m_y)\leq\left(\frac12\dist(x_0,y_0)^p+\frac12\dist(x_1,y_1)^p\right)^{1/p}
    \]for some $p\in[1,\infty]$. If $X$ satisfies the $p$--Busemann condition, we say that $(X,\dist)$ is \textit{$p$--Busemann}. In particular, if $p=1$, we say that $(X,\dist)$ is \textit{Busemann}.

    It turns out that $(X,\dist)$ is a Busemann space if and only if \[
    \dist(m(x,z),m(x,y))\leq\frac12\dist(z,y).
    \]
    \item $X$ is \textit{uniformly $p$--convex} for some $p\in[1,\infty]$ if, for all $\epsilon>0$, there exists $\rho_p(\epsilon)\in(0,1)$ such that, for every $x,y,z\in X$ satisfying \[
    \dist(x,y)>\epsilon\left(\frac12\dist(x,z)^p+\frac12\dist(y,z)^p\right)^{1/p}\text{, for some }p>1,
    \]or \[
    \dist(x,y)>|\dist(x,z)-\dist(y,z)|+\epsilon\left(\frac12\dist(x,z)+\frac12\dist(y,z)\right)\text{, for }p=1,
    \]the following inequality holds: \[
    \dist(m(x,y),z)\leq(1-\rho_p(\epsilon))\left(\frac12\dist(x,z)^p+\dist(y,z)^p\right)^{1/p}.
    \]For example, every $\mathrm{CAT}(0)$--space is uniformly $2$--convex.
\end{enumerate}
\end{definition}

\begin{rema}
    By \cite[Lemma 1.4., Corollary 1.5.]{kell}, the following assertions hold: \begin{itemize}
        \item A uniformly $p$--convex metric space is uniformly $p'$--convex for all $p'\geq p$.
        \item Assume $(X,\dist)$ is Busemann. Then $(X,\dist)$ is strictly (resp. uniformly) $p$--convex for some $p\in[1,\infty]$ if and only if it is strictly (resp. uniformly) $p$--convex for all $p\in[1,\infty]$.
        \item Any $\mathrm{CAT}(0)$--space is both Busemann and uniformly $2$--convex, thus uniformly $p$--convex for every $p\in[1,\infty]$.
    \end{itemize}
\end{rema}

In order to apply some of Kell's results, we introduce the notion of \textit{reflexivity} on metric spaces.

\begin{definition}[Reflexive metric space, Definition 2.1. \cite{kell}]{\label{kelldefinitions2}}
    Let $I$ be a directed set. A metric space $(X,\dist)$ is \textit{reflexive} if, for every non--increasing family $\{C_i\}_{i\in I}\subset X$ of non--empty bounded closed convex subsets (i.e. $C_i\subset C_j$ whenever $i\geq j$), we have \[
    \bigcap_{i\in I}C_i\neq\emptyset.
    \]
\end{definition}

\section{Wasserstein type spaces and distances}{\label{wassersteintype}}

In this section we will recall the standard notions of optimal transport and Wasserstein distance. Then we will provide an introduction to the Orlicz--Wasserstein spaces initially proposed by Sturm in \cite{sturm}. Finally, we present our last Wasserstein-type space, the one formed by persistence diagrams, the key element in the field of Topological Data Analysis \cite{chazalintro}.

\subsection{Wasserstein space}

From now on, $X$ will be a metric space with distance function $\dist$. Denote by $\mathcal P(X)$ the set of probability measures on $X$ and by $\mathcal P _p (X)$ the probability measures with finite $p$-moment, i.e.
\[
\mathcal P_p(X) := \{ \mu \in \mathcal P(X) : \int_X \dist(x, x_0)^p d \mu(x) < \infty   \text{ for some $x_0\in X$} \}.
\]

\begin{definition}[Transference plan]
A \emph{transference plan} between two positive measures $\mu, \nu \in \mathcal P(X)$ is a finite positive measure $\pi \in \mathcal P (X \times X)$ which satisfies that, for all Borel subsets $A, B$ of $X$,
\[ \pi(A\times X) = \mu(A), \quad \text{and} \quad \pi(X \times B) = \nu(B). \]
\end{definition}
Note that we require $1=|\mu| = | \nu| = \pi( X \times X)$, so we are not considering all measures of the product space. We denote by $\Gamma(\mu, \nu)$ the set of transference plans between the measures $\mu$ and $\nu$. Then, we define the $p$--Wasserstein distance for $p\ge 1$ between two probability measures as
\[
W_p( \mu, \nu) := \left( \min_{\pi \in \Gamma(\mu, \nu)} \int_{X \times X }\dist(x,y)^p d\pi (x, y) \right) ^{\frac{1}{p}} .
\]
The metric space $(\mathcal P_p (X) , W_p)$ is denoted as the \emph{$p$--Wasserstein space of $X$.}

It is easy to see that, for $x, y \in X$, $W_p(\delta_x, \delta_y) = \dist(x,y)$. Therefore the inclusion $x \mapsto \delta_x$ is an isometric embedding of $X$ inside $W_p(X)$. 

\begin{rema}
    When we are calculating $W_p(\delta_x, \mu)$, there exists only one pairing $\pi = \delta_x \otimes \mu \in \Gamma(\delta_x, \mu)$ between a delta and a general probability measure. Therefore the Wasserstein distance can be easily computed by
    \[
    W_p^p(\delta_x, \mu) = \int_X \dist(x, y) ^p d\mu(y).
    \]
    In particular, fixing $a, x, y\in X$, and $0 \le \lambda \le 1$, we have
    \[
    W_p^p(\delta_a, \lambda \delta_x + (1-\lambda) \delta_y) = \lambda \dist(a, x)^p + (1-\lambda)
    \dist (a, y) ^p.
    \]
\end{rema}

\subsection{Orlicz--Wasserstein space}\label{subseccionorlicz}

Let $\vartheta: \RR^+ \to \RR^+$ be a strictly increasing, continuous function. Assume $\vartheta$ admits a representation $\vartheta = \varphi \circ \psi$ as a composition of a convex and a concave function $\varphi$ and $\psi$, respectively. This includes all $\mathcal C^2$ functions \cite[Example 1.3.]{sturm}.

\begin{definition}[$L^\vartheta$--Wasserstein space and distance]
    Let $(X, \dist)$ be a complete separable metric space. The $L^\vartheta$--Wasserstein space $\mathcal P_\vartheta (X) $ is defined by all probability measures $\mu$ in $X$ such that
    \[ \int_X \varphi \left(\frac{1}{t} \psi(\dist(x,y))\right) d \mu(x) < \infty.\]
    The $L^\vartheta$--Wasserstein distance of two probability measures $\mu, \nu \in \mathcal P_\vartheta (X)$ is defined as
    \[
    W_\vartheta (\mu, \nu) = \inf \left\{ t>0 : 
    \inf_{\pi \in \Gamma(\mu, \nu) } \int_{X\times X} \varphi \left( \frac1t \psi(\dist(x,y))\right) d\pi(x,y)  \le 1  \right\}.
    \]
\end{definition}
The function $W_\vartheta$ is a complete metric on $\mathcal P_\vartheta(X)$ (see \cite{sturm}, Proposition 3.2).
The metric space $(\mathcal P_\vartheta (X), W_\vartheta)$ is known as the \textit{$\vartheta$-Orlicz--Wasserstein space of $X$}.

Notice that for every $x\in X$, the probability measure $\delta_x$ belongs to $\mathcal P_\vartheta(X)$. Therefore, we can embed the metric space $X$ inside its Orlicz--Wasserstein space by mapping $x \mapsto \delta_x$. In addition, this map is an isometric embedding if and only if $\psi \equiv \operatorname{Id}$ and $\varphi(1) = 1$.

\subsection{Space of persistence diagrams and bottleneck distance}

We will now define $(\dgm_p,w_p)$ the space of persistence diagrams with a Wasserstein metric. For that purpose we being with the basic notion of the elements of our metric space:

\begin{definition}[Persistence diagram, \cite{bubenik}]
    A \textit{persistence diagram} is a function from a countable set $I$ to $\RR^2_<$, i.e. $D\colon I\to\RR^2_<$, where $\RR^2_<=\{(x,y)\in\RR^2\colon x<y\}$.
\end{definition}

\begin{rema}
    In this definition, all the points have multiplicity one. Other authors suggest considering persistence diagrams as multisets of points, i.e. sets of points where we can repeat points (see \cite{guijarro1,guijarro2, Mileyko2011,turnerfrechet}). This consideration is closer to the performance of the persistence diagrams in the TDA setting as various homological features can have the same birth and death.  
    Also, in \cite{guijarro1,guijarro2}, the authors extend the notion of persistence diagrams beyond the Euclidean setting and present a general definition for points in metric spaces.
\end{rema}

Once we have the points of our metric space, we want to define distance functions on it. For that purpose, first, we present two preliminary definitions.

\begin{definition}[Partial matching, \cite{bubenik}]
    Let $D_1\colon I_1\to\RR^2_<$ and $D_2:I_2\to\RR_<^2$ be persistence diagrams. A \textit{partial matching} between them is a triple $(I_1',I_2',f)$ such that $I_1'\subseteq I_1$, $I_2'\subseteq I_2$, and $f\colon I_1'\to I_2'$ is a bijection.
\end{definition}

In the same spirit of the original $p$--Wasserstein distance, we want to define a new one between persistence diagrams $D_1$ and $D_2$ as the minimal cost of a partial matching between them. In particular, the cost of a partial matching will be the $\ell^p$ norm of distances between matched pairs and the distances between unmatched pairs and $\Delta$, as Bubenik stated on \cite[Section 2.1.]{bubenik} , where $\Delta$ denotes the diagonal in $\RR^2$.

\begin{definition}[Cost of a partial matching, \cite{bubenik}]
    Let $D_1\colon I_1\to\RR^2_<$ and $D_2\colon I_2\to\RR^2_<$ be persistence diagrams and $(I_1',I_2',f)$ a partial matching between them. We endow $\RR^2$ with the infinity metric $\dist_\infty(a,b)=\norm{a-b}_{\infty}=\max(|a_x-b_x|,|a_y-b_y|)$. Observe that, for $a\in\RR^2_<$, we have that $\dist_\infty(a,\Delta)=\inf_{t\in\Delta}\dist_\infty(a,t)=(a_y-a_x)/2$. We denote by $\cost_p(f)$ the \textit{$p$--cost of $f$}, defined as follows. For $p<\infty$, let \[
    \cost_p(f)=\left(\sum_{i\in I_1'}\dist_\infty(D_1(i),D_2(f(i)))^p+\sum_{i\in I_1\backslash I_1'}\dist_{\infty}(D_1(i),\Delta)^p+\sum_{I_2\backslash I_2'}\dist_\infty(D_2(i),\Delta)^p\right)^{1/p},
    \]and for $p=\infty$, let \[
    \cost_\infty(f)=\max\left\{\sup_{i\in I_1'}\dist_\infty(D_1(i),D_2(f(i))),\sup_{i\in I_1\backslash I_1'}\dist_\infty(D_1(i),\Delta),\sup_{i\in I_2\backslash I_2'}\dist_\infty(D_2(i),\Delta)\right\}.
    \]
    If any of the terms in either expression is unbounded, we declare the cost to be infinity.
\end{definition}

Now we can define the distance functions and the metric space of persistence diagrams:

\begin{definition}[$p$--Wasserstein distance and bottleneck distance of persistence diagrams, \cite{cohenbottleneck}]
    Let $1\leq p\leq\infty$ and $D_1$, $D_2$ persistence diagrams. Define \[
    \tilde{w}_p(D_1,D_2)=\inf\{\cost_p(f)\colon f\text{ is a partial matching between }D_1\text{ and }D_2\}.
    \]Let $(\dgm_p,w_p)$ denote the metric space of persistence diagrams $D$ such that $\tilde{w}_p(D,\emptyset)<\infty$ with the relation $D_1\sim D_2$ if $\tilde{w}_p(D_1,D_2)=0$, where $\emptyset$ is the unique persistence diagram with empty indexing set. The metric $w_p$ is called the \textit{$p$--Wasserstein distance} and $w_\infty$ is called the \textit{bottleneck distance}.
\end{definition}

\section{Reach of the Wasserstein space}{\label{reachwasserstein}}

The $p$--Wasserstein space of a metric space $(X,\dist)$ gives us an example of a well-known isometric embedding between a total space and an infinite dimensional space (in this sense, there also exists the Kuratowski embedding between a compact metric space $(Y,\dist_Y)$ and the space of functions $L^\infty(Y)$).


We first recall the definition of the \textit{set of unique points} and \textit{reach}:

\begin{definition}[Unique points set and reach, \cite{federer}]
    Let $(X, \dist)$ be a metric space and $A\subset X$ a subset. We define the set of points having a unique metric projection in $A$ as
    \[\unp(A) = \{x \in X : \text{there exists a unique $a$ such that }  \dist(x,A) = \dist(x,a)\}.\]
    For $a\in A$, we define the \emph{reach} of $A$ at $a$, denoted by $\reach(a, A)$, as
    \[\reach(a, A) = \sup \{ r\ge 0 : B_r(a) \subset \unp(A)\}.\]
    Finally, we define the \emph{global reach} by 
    \[\reach(A) = \inf_{a\in A} \reach(a, A).\]
\end{definition}

The set of unique points of the isometric embedding of a metric space into its $p$--Wasserstein space is dense into the total space:

 \begin{prop}
        Let $(X,\dist)$ be a non-branching metric space and $W_p(X)$ with $p>1$ its $p$--Wasserstein space. Then the set of unique points $\unp(X\subset W_p(X))$ is dense in $W_p(X)$.
    \end{prop}

    \begin{proof}
         Let $\mu \in W_p(X)$ be a measure with $x \in X$ a barycenter. Take $\nu$ inside a  geodesic between $\delta_0$ and $\mu.$ Suppose that there exists some other point $z\in X$ that is as barycenter for $\nu.$ This implies that $W_p(\nu, \delta_z) \leq W_p(\nu, \delta_x)$ and with this we get
        \[W_p(\mu,\delta_z)\leq W_p(\mu,\nu)+W_p(\nu,\delta_z)\leq W_p(\mu,\nu)+W_p(\nu,\delta_x)= W_p(\mu,\delta_x).  \]
        So then $z$ is also a barycenter for $\mu.$ Furthermore we notice that there is a branching geodesic joining $\mu$ with $\delta_z.$ This gives us the contradiction as $W_p(X)$ is non-branching.
        Then $\nu$ is a measure in $\unp(X\subset W_p(X))$ which can be taken arbitrarily close to $\mu.$
    \end{proof}

    This density fact motivates the question of the existence of metric spaces with positive reach into its $p$--Wasserstein space.

\subsection{Null reach}
The first result of this paper is that the reach of a metric space inside its $1$--Wasserstein space is always 0. The proof follows the idea of the proof of Theorem 1.6. of \cite{cuerno}.

\begin{thm}\label{reach1wasserstein}
   Let $(X,\dist)$ be a metric space, and consider its $1$--Wasserstein space, $W_1(X).$ Then, for every accumulation point $x\in X$, $\reach(x, X\subset W_1(X)) = 0$. In particular, if $X$ is not discrete, $\reach(X\subset W_1(X)) = 0$.
\end{thm}

\begin{proof}
    Following \cite{cuerno}, let $\epsilon>0$. We will show that inside $B_\epsilon(x)\subset W_1(X)$ there exists at least one measure $\mu\notin\unp(X)$.  
    
    By hypothesis, there exists $y\in X$, $y\neq x$, such that $d(x,y)<\epsilon$. Then \[
    \mu:=\frac12\delta_x+\frac12\delta_y.
    \]First, notice that $\mu\neq\delta_z$ for any $z\in X$ because the support of $\mu$ is different from the support of any of the $\delta_z\in X$.
    In addition, due to the triangle inequality, \begin{equation}{\label{puntomedio}}
    W_1(\delta_a,\mu)=\frac{1}{2}\dist(a,x)+\frac12\dist(a,y)\geq\frac12\dist(x,y).
    \end{equation}

    By inequality \eqref{puntomedio} above, we can clearly see that $\mu\in B_\epsilon(x)$, because \[
    W_1(\delta_x,\mu)=\frac12\dist(x,y)<\epsilon.
    \]

    Finally, we observe that both $a=x$ and $a=y$ minimize the distance to $\mu$. Therefore, $\mu\notin \unp (X)$ and $\reach(x, X\subset W_1(X)) = 0$.
\end{proof}
Note that the hypothesis of the point being an accumulation point is necessary, because, if $x_0\in X$ is an isolated point, then the quantity $\ell = \inf_{x\in X} \dist(x, x_0)$ is strictly positive, and $B_{\ell/2}(x)$ admits a unique metric projection to $X$.

An interesting observation is that, combining the same argument in the proof of Theorem \ref{reach1wasserstein} with the previous remark, if $X$ is a discrete metric space isometrically embedded into another metric space $Y$, then  $\reach(X \subset Y) = \inf_{x_1\neq x_2} \dist(x_1, x_2)/2>0$.

Now we will provide results about the reach of a geodesic metric space inside its $p$--Wasserstein space with $p>1$. We have found that these results are closely related to the uniqueness of the geodesics. This next proposition has important consequences about the reach inside a Wasserstein space, as it constructs measures with possibly several projections in $X$.

\begin{prop}\label{proposicion21} 
Let $(X,\dist)$ be a geodesic metric space, and $x, y\in X$ two points with $x\neq y$.
Consider the probability measure $\mu= \lambda \delta_x + (1-\lambda) \delta_y$, for $0<\lambda<1$. Then $\mu$ minimizes its $p$--Wasserstein distance to $X$ exactly once for every minimizing geodesic between $x$ and $y$. 
\end{prop}
\begin{proof}
    The proof is structured in the following way: First, we choose a candidate for the distance--minimizer of $\mu$, supposing it lies inside a minimizing geodesic. Then, we show that the global minimum distance can only be achieved inside a minimizing geodesic.
    
    Choose $\gamma(t) \colon [0,1] \to X$ a minimizing geodesic from $x$ to $y$. We can compute the cost $W_p^p( \delta_{\gamma(t)}, \mu)$ and then minimize in $t$. Indeed, 
    \begin{equation}{\label{ecuacionwassers}}
W_p^p(\delta_{\gamma(t)}, \mu) = \lambda \dist(\gamma(t), x)^p + (1-\lambda) \dist(\gamma(t), y) ^p = (\lambda t^p + (1-\lambda) (1-t)^p  )\dist( x, y)^p.
\end{equation}
The minimum will be achieved at the parameter $t_0$ which verifies $\dfrac{d}{dt}\bigg|_{t=t_0} W_p^p (\delta_{\gamma(t)}, \mu) = 0$. We know this because that derivative is negative for $t=0$, and positive for $t=1$, and vanishes at only one point $t=t_0$. An easy computation shows us that the only solution in our interval is
\[
t_0 = \frac{(1-\lambda)^{p-1}}{\lambda^{p-1}+(1-\lambda)^{p-1}}.
\]
Thus, the Wasserstein distance between $\mu$ and this \textit{geodesic minimum} is
\[W_p^p(\delta_{\gamma(t_0)}, \mu ) = \frac{\lambda(1-\lambda)^{(p-1)p} + (1-\lambda) \lambda^{(p-1)p}}{(\lambda^{p-1}+ (1-\lambda)^{p-1})^p} \cdot \dist^p(x,y).\]
Observe that this value is independent from the minimizing geodesic $\gamma$ of our choice.

Finally, we only have to prove that the minimum can only be achieved inside a minimizing geodesic. For that purpose, we will choose any $a \in X$, and we will construct another point $a'$ inside a minimizing geodesic segment $\gamma$ verifying $W_p^p(\delta_{a}, \mu) \ge W_p^p(\delta_{a'}, \mu) $.

The case $\dist(a, y) \ge \dist(x, y)$ is straightforward,  as choosing $a' = x$ we have 
\begin{align*}
    W_p^p (\delta_a, \mu) &= \lambda \dist(a, x) ^p + (1-\lambda) \dist(a, y) ^p\\
    &\ge (1-\lambda) \dist(a, y) ^p \\
    &\ge (1-\lambda) \dist(x, y) ^p = W_p^p (\delta_x, \mu).
\end{align*}

Now, if $\dist(a, y) < \dist(x, y)$, we can pick $a'$ inside $\gamma$ at distance $\dist(a, y)$ to $y$. Observe that $\dist(a, x) \ge \dist(a', x)$ or $\gamma$ would not be minimizing. Then
\begin{align*}
W_p^p (\delta_a, \mu) &= \lambda \dist(a, x) ^p + (1-\lambda) \dist(a, y) ^p \\
&= \lambda \dist(a, x) ^p + (1-\lambda) \dist(a', y) ^p \\
&\ge  \lambda \dist(a', x) ^p + (1-\lambda) \dist(a', y) ^p = W_p^p (\delta_{a'}, \mu).
\end{align*}
Therefore, the minimum can only be achieved inside minimizing geodesics between $x$ and $y$ and our proof is complete.
\end{proof}

Now, we will apply the preceding proposition to construct measures with multiple projections close to any point in $X$. We will use this to derive sufficient conditions for attaining $\reach(p,X)=0$ for all $p\in X$.

\begin{thm}\label{teorema22}
 Let $X$ be a geodesic metric space, and $x\in X$ a point such that there exists another $y\in X$ with the property that there exist at least two different minimising geodesics from $x$ to $y$. Then, for every $p>1,$ \[
\reach(x, X\subset W_p(X))=0.
\]    
In particular, if there exists a point $x\in X$ satisfying that property, $\reach(X\subset W_p(X))=0$ for every $p>1$.
\end{thm}

\begin{proof}
The probability measure $\mu_\lambda = \lambda \delta_x + (1-\lambda )\delta_y$ will have at least two different points minimizing its distance to $X$ by proposition \ref{proposicion21}. 

    Now simply observe that $W_p^p(\mu_\lambda, \delta_x) = (1-\lambda) \dist(x, y)^p$, which decreases to 0 when $\lambda \to 1 $. Hence $\reach(x, X) = 0$ for every $x\in X$ satisfying that property, and therefore $\reach(X\subset W_p(X))=0$.    
\end{proof}
When $X$ is a Riemannian manifold, some common hypothesis will grant us reach $0$.
For example, a classic result by Berger (see for example \cite[Chapter 13, Lemma 4.1]{docarmo}) proves that our theorem can be applied when $X$ is compact. In this case, for any $p\in X$, there always exists another $q\in X$ such that there exist two minimizing geodesics starting at $p$ to $q$. More precisely, for every $p\in X$ we can choose a maximum $q$ of the function $\dist(p, \cdot)$ and there will be at least two minimal geodesics from $p$ to $q$. There is a similar result in \cite{luisyfernando}, where it is shown that for every $p$, there exists $q\in X$ such that $p$ and $q$ are joined by several minimizing geodesics.

\begin{cor}{\label{corollarycompact}} 
If $M$ is a compact Riemannian manifold, then $\reach(x, M\subset W_p(M)) =0$ for every $p>1$ and $x\in M$.
\end{cor}

Also, we can apply our Theorem \ref{teorema22} to the non simply connected case:

\begin{cor}{\label{corollarynotsimply}}
    If $M$ is a complete Riemannian manifolds with non--trivial fundamental group (i.e. not simply connected), then $\reach(x, M\subset W_p(M)) =0$ for every $p>1$ and $x\in M$.
\end{cor}
\begin{proof}
    Consider the universal cover $\pi \colon \tilde{M} \to M$. Let $x\in M$, and let $\tilde{x}$ be a point with $\pi(\tilde{x}) = x$. Denote by $G$ the fundamental group of $M$. We know that $G$ acts on $\tilde{M}$ by isometries and that $G\tilde{x}$ is a discrete, locally finite set. Then, we may take $\tilde{x}'\in G\tilde{x}$ at minimal distance from $\tilde{x}$. 

    Then we can take a minimizing geodesic $\tilde{\gamma}: [0, \ell] \to \tilde{M}$ from $\tilde{x}$ to $\tilde{x}'$, and the projection $\gamma= \pi \circ \tilde{\gamma}$ will be a geodesic loop such that $\gamma(0)= \gamma(\ell) = x$, and $\gamma$ is globally minimizing on $[0, \ell/2]$ and $[\ell/2, \ell]$. Otherwise, by taking a shorter curve to the midpoint $\gamma(\ell/2)$ and lifting it we could construct a shorter geodesic from $\tilde{x}$ to another point in $G\tilde{x}$ and our two points would not be at minimal distance.
\end{proof}

\subsection{Infinite reach}


For this subsection, we will use results obtained by Kell \cite{kell}, employing the metric definitions presented in Section \ref{metricdefinitions}. The combination of these elements yields results that imply infinite reach for certain metric spaces.

\begin{thm}{\label{reachpositivowass}}
Let $(X,\dist)$ be a reflexive metric space. Then the following assertions hold:
    \begin{enumerate}
        \item If $X$ is strictly $p$--convex for $p\in[1,\infty)$ or uniformly $\infty$--convex if $p=\infty$, then\begin{equation}
            \reach(X\subset W_r(X))=\infty\text{, for }r>1.
        \end{equation} 
        \item If $X$ is Busemann, strictly $p$--convex for some $p\in[1,\infty]$ and uniformly $q$--convex for some $q\in[1,\infty]$, then\begin{equation}
            \reach(X\subset W_r(X))=\infty\text{, for }r>1.
        \end{equation}
    \end{enumerate}
\end{thm}

\begin{proof}

In \cite[Theorem 4.4.]{kell}, Kell establishes that any $p$-convex and reflexive metric space possesses $p$-barycenters, as he defined them in  \cite[Definition 4.3.]{kell}. His Theorem 4.4. establishes the existence of such barycenters but not uniqueness. To establish it, we require the conditions we stated in both cases of our theorem. Now we present how these restrictions give us infinity reach.

    \begin{enumerate}
        \item Following \cite[Corollary 4.5.]{kell}, the spaces $(X,\dist)$ which satisfy the hypotheses in item (1) of the theorem have unique $r$--barycenters for $r>1$. In other words, every $\mu\in W_r(X)$ has a unique barycenter. This finishes the proof of the first assertion of the theorem.
        \item Following \cite[Lemma 1.4.]{kell}, if $(X,\dist)$ is strictly (resp. uniformly) $p$--convex for some $p$, then it is strictly (resp. uniformly) $p$--convex for all $p$. Hence, we are in the case (1). \qedhere
    \end{enumerate}
\end{proof}

As we pointed in Section \ref{metricdefinitions}, $\mathrm{CAT}(0)$--spaces are a well--known example of metric spaces satisfying some of the hypotheses in Theorem \ref{reachpositivowass}. In that sense, there is a straightforward corollary to our Theorem \ref{reachpositivowass} in terms of $\mathrm{CAT}(0)$--spaces:

\begin{cor}{\label{reachinfinitocat0}}
    Let $(X,\dist)$ be a reflexive $CAT(0)$--space, then  \[
    \reach(X\subset W_p(X))=\infty, \text{ for }p>1.
    \]
\end{cor}

\begin{proof}
    As Kell stated in \cite[Last line of Introduction]{kell}, $\mathrm{CAT}(0)$-spaces are both Busemann spaces and uniformly $p$--convex for every $p\in[1,\infty]$. 
    Moreover, from the definition of $\mathrm{CAT}(k)$-spaces, with $k=0$, we have that
    \begin{align*}
    \dist(m(x,y),z)&\leq\dist_{\mathbb{E}^n}(m(x',y'),z')\\
        &\frac12\left(\dist_{\mathbb{E}^n}(x',z')+\dist_{\mathbb{E}^n}(y',z')\right)=\frac12\left(\dist(x,z)+\dist(y,z)\right).
    \end{align*}
Hence, $\mathrm{CAT}(0)$--spaces are strictly $1$--convex and, by \cite[Lemma 1.4]{kell} they are strictly $p$--convex for all $p$. The conclusion now follows from item (2) in Theorem \ref{reachpositivowass}.
\end{proof}

\begin{rema}
It is easy to check, from the definition, that $\mathrm{CAT}(0)$--spaces are contractible, and, therefore, simply connected. This is a necessary condition for Theorem \ref{reachpositivowass}, because if this were not the case, we would have a closed geodesic and Proposition \ref{proposicion21} would give us zero reach for the points inside that geodesic.
\end{rema}

As particular cases of $\mathrm{CAT}(0)$--spaces, we have Hadamard manifolds (complete, simply connected Riemannian manifolds with non-positive sectional curvature everywhere) and, in particular, Euclidean $n$--space. So, as a corollary, we obtain the following:

\begin{cor}\mbox{}%

\begin{enumerate}
    \item Let $(M^n,g)$ be a Hadamard manifold. Then \[
    \reach(M^n\subset W_p(M^n))=\infty, \textit{ for }p>1.
    \]
    \item Let $\mathbb{E}^n$ be the Euclidean $n$--space. Then \[
    \reach(\mathbb{E}^n\subset W_p(\mathbb{E}^n))=\infty, \textit{ for }p>1.
    \]
\end{enumerate}
\end{cor}

Other authors have considered the existence of  barycenters in the $\mathrm{CAT}(\kappa)$--space context, specifically $\kappa=0$. In \cite[Proposition 4.3.]{sturmnonpositive}, Sturm proved the existence and uniqueness of barycenters for $\mathrm{CAT}(0)$--spaces only for the $2$--Wasserstein space. In \cite[Theorem B]{yokota}, Yokota stated a condition on $\mathrm{CAT}(\kappa)$--spaces, with $\kappa>0$, to have unique barycenters. This condition is related to the size of the diameter of the $\mathrm{CAT}(\kappa)$--space, which needs to be small in order to have unique barycenters.

\subsection{Projection map}

The infinity of the reach leads to a natural question about the regularity of the \textit{projection map}, i.e., \begin{align*}
    \proj_p: W_p(X)&\to X\\
    \mu&\mapsto r_\mu,
\end{align*}where $r_\mu\in X$ denotes the barycenter of the measure $\mu$, that is, the only point in $X$ that minimizes the distance to $\mu$.

Let $(X,\dist)$ be a metric space for which Theorem \ref{reachpositivowass} holds. Then $X$ has infinity reach; in other words, every measure has a unique barycenter and $\proj_p$ is well--defined. Moreover, the fibres of the map are convex. Let $\mu$, $\nu\in\{\sigma\in W_p(X)\colon r_\sigma=a\}$, $\lambda\in(0,1)$ and $b\in X$. Then \[
W_p^p(\lambda\mu+(1-\lambda)\nu,\delta_b)=\lambda W_p^p(\mu,\delta_b)+(1-\lambda)W_p^p(\nu,\delta_b)\geq\lambda W_p^p(\mu,\delta_a)+(1-\lambda)W_p^p(\nu,\delta_a),
\]since $\mu$, $\nu\in\{\sigma\in W_p(X)\colon r_\sigma=a\}$.

A \textit{submetry} between two metric spaces $X, Y,$ is a map $f\colon Y \to X$ such that, for every $a\in Y$ and every $r\ge 0$, we have $f(B_Y(a, r)) = B_X(f(a), r)$. For more information about this type of maps, we refer the reader   \cite{beressub,guijarrobere,guijarrosub,kaplytchaksub}.

We briefly recall Kuwae's property \textbf{B}, (see Section $4.3$ in \cite{kell}  and references therein). Take  two geodesics $\gamma, \eta$ such that they intersect at an unique point $p_0.$ Assume that for all points  $z \in \gamma [0,1]$
the minimum of the map $t \mapsto \|z-\eta_t\|$ is achieved only by the point $p_0.$ Then for every point $w \in \eta[0,1]$ the minimum of the map $t \mapsto \|w-\gamma_t\|$ is achieved only by $p_0.$

\begin{thm}{\label{thmsubmetry}}
 Let $(X,\|\cdot\|)$ be a reflexive Banach space equipped with a strictly convex norm and satisfying property \textbf{B}. Then  $proj_2$ is a submetry.
\end{thm}

\begin{proof}
 First let us make a couple observations. From the strict convexity of the norm it follows that between any two points $x,y \in X$ there is a unique geodesic joining them, more precisely it is the curve $[0,1]\ni t \mapsto (1-t)x+ty.$

In particular this tells us that 
$m(x,y)= \frac{1}{2}x+\frac{1}{2}y.$  

Let $p>1$ and $x,y,z\in X$ then
 \begin{align*}
 \|m(x,y)-z\|^p &= \|\frac{1}{2}x+\frac{1}{2}y-z\|^p \\
 &< 2^{p-1}\left(\|\frac{1}{2}(x-z)\|^p+\|\frac{1}{2}(y-z)\|^p\right)\\
 &= 2^{p-1}\left(\frac{1}{2^p}\|x-z\|^p+\frac{1}{2^p}\|y-z\|^p\right)\\
 &= \frac{1}{2}\|x-z\|^p+\frac{1}{2}\|y-z\|^p
 \end{align*}
 Hence $(X,\|\cdot\|)$ is strictly $p-$convex and so it satisfies the conditions of Theorem \ref{reachpositivowass}, with this barycenters exist and are unique. Therefore the projection map $\proj_2$ is well defined.

 Now notice that 

 \begin{align*}
 \|m(x,y)-m(y,z)\| &= \|\frac{1}{2}(x+z)-\frac{1}{2}(y+z)\|\\
 &= \frac{1}{2}\|x-y\|.
 \end{align*}
Which implies that for $p>1$
\[\|m(x,y)-m(y,z)\|^p < \frac{1}{2}\|x-y\|^p, \]
i.e. it is $p-$Busemann. Then the  
$2-$Jensen inequality  (see Section $4.3$   in \cite{kell}) holds and so in addition  we have that by Proposition $4.8$ in \cite{kell}  $\proj_2$ is $1-$Lipschitz. Let $B_r(\mu)$ be a ball in the Wasserstein space. We just proved  that 
\[\proj_2(B_r(\mu) )\subset B_r(\proj_2(\mu)).\]
Then, it suffices to see that every point in $B_r(\proj_2(\mu))$ is the image of a point (the barycenter of a measure) in $B_r(\mu)$. Fix $\mu, r\ge 0$ and let $b\in B_r(\proj_2(\mu))$. Let $T$ be the translation from $\proj_2(\mu)$ to $b$. Let us show that $T_\# \mu$ has $b$ as a barycenter. For any $a\in X$,
    \begin{align*}
    W_2^2(T_\# \mu, \delta_{T(a)}) &= \int_{X} \|x-T(a)\|^2 \, d(T_\# \mu )(x)\\ 
    &= \int_{X} \|T(x)-T(a)\|^2 \, d\mu (x)
    \\ &= \int_{X} \|x-a\|^2 \,d\mu(x) = W_2^2(\mu, \delta_a).
    \end{align*}
    Hence, if $a=\proj_2(\mu)$, then $a$ minimizes the distance from $X$ to $\mu$, and then $T(a)=b$ minimizes the distance to $T_\#\mu$. 

    It remains to see that $T_\#\mu$ is contained in $B_r(\mu)$. Choosing $(\operatorname{Id}, T)_\# \mu$ as a transport plan in $\Pi(\mu, T_\#\mu)$, 
    \begin{align*}
    W_2^2(\mu, T_\#\mu) &= \inf_{\pi \in \Pi(\mu, T_\#\mu)} \int_{X \times X} \|x-y\|^2 \, d\pi(x,y) \\ &\le \int_{X} \|x-T(x)\|^2 \, d\mu(x) = \| \proj_2(\mu)-b\|^2 < r^2.
    \end{align*}
    Therefore, $T_{\#}\mu\in B_r(\mu).$
\end{proof}

Examples of spaces satisfying the assumptions of Theorem  include Hilbert spaces and $L^p$ spaces (see Examples $4.5,$ and $4.6$ in \cite{kuwae}).

\section{Reach of the Orlicz--Wasserstein space}{\label{reachwassersteinorliz}}
An introduction to Orlicz--Wasserstein spaces can be found in subsection \ref{subseccionorlicz}. More information about this type of spaces along with a proof of their completeness can be found in \cite{sturm}.
\subsection{Null reach}

We start this section with a simple remark. 
\begin{rema}
Let $\varphi\equiv Id$. Observe that $\psi \circ \dist$ is a distance when $\psi$ is a positive concave function with $\psi(0)=0$. Then $W_\vartheta$ is a 1-Wasserstein distance for the metric space $(X, \psi \circ \dist)$.  Therefore, \[
        \reach(x, X\subset W_{\vartheta}(X))=0
        \]
        whenever $x\in X$ is an accumulation point, by Theorem \ref{reach1wasserstein}.
\end{rema}

We can replicate Proposition \ref{proposicion21} for the case where $X$ is isometrically embedded into an Orlicz--Wasserstein space using a more delicate argument.

\begin{prop}\label{proposicion33} 
Let $X$ be a geodesic metric space, and let $x, y\in X$ be two points with $x\neq y$.
Consider the probability measure $\mu= \lambda \delta_x + (1-\lambda) \delta_y$, for $0<\lambda<1$. Then, the following assertions hold: \begin{enumerate} 
\item $\mu$ can only minimize its $\vartheta$--Wasserstein distance to $X$ inside a minimizing geodesic between $x$ and $y$.
\item If $\lambda$ is close to one, and there exists a constant $c>1$ such that $\varphi^{-1}(t)< t $ for every $t>c$, then the minimum will be attained inside the interior of each geodesic.
\end{enumerate}
\end{prop}
\begin{proof}
    First we will see that the minimum can only be attained inside a geodesic. For that purpose, we will replicate the argument in the proof of Proposition \ref{proposicion21}. That is, given $a\in X$, we construct $a'\in\gamma([0, \ell])$, where $\gamma$ is a minimizing geodesic, with 
    \[
    W_\vartheta(\delta_a, \mu) >  W_\vartheta(\delta_{a'}, \mu) .
    \]
    Again, it suffices to consider the case $\dist(a, y) \le \dist(x, y)$. We can pick $a' \in \gamma([0,\ell])$ such that $\dist(a', y) = \dist(a, y)$. Then, $\dist(a', x) < \dist(a, x)$ or $a$ is also inside a minimizing geodesic.

    Let \[
    S=\left\{ t> 0: \lambda \varphi\left( \frac{1}{t} \dist(a, x) \right) + (1-\lambda) \varphi \left( \frac{1}{t} \dist(a, y) \right)  \le 1 \right\}.
    \]As we have only one transport plan $\pi = \delta_a \otimes \mu$, we can write
    \[
    W_\vartheta(\delta_a, \mu) = \inf S.
    \]
    Thus, it is enough to see that, if $t_0$ verifies the inequality inside that infimum for $a$, then it will verify it for $a'$. Indeed,
    \begin{align*}
        1 &\ge \lambda \varphi\left( \frac1{t_0}\dist(a, x) \right) + (1-\lambda) \varphi \left( \frac{1}{{t_0}} \dist(a, y) \right) \\ &= \lambda \varphi\left( \frac1{t_0} \dist(a, x) \right) + (1-\lambda) \varphi \left( \frac{1}{{t_0}} \dist(a', y) \right)\\
        &> \lambda \varphi\left( \frac1{t_0}\dist(a', x) \right) + (1-\lambda) \varphi \left( \frac{1}{{t_0}} \dist(a', y) \right).
    \end{align*}
    The last inequality comes from the monotonicity of $\varphi$, and the assumption $\dist(a', x) < \dist(a,x)$. Observe that, because the previous inequality is strict, we will have a strict inequality in $W_\vartheta(\delta_a, \mu) >  W_\vartheta(\delta_{a'}, \mu) $.

    Now we will prove the second part of our proposition. Assuming $\lambda$ close to 1, and that $\varphi$ differs from the identity for big enough values, we will see that there are points $a \in \gamma((0, \ell) )$ with 
    \begin{equation}\label{ecuacionprop33}
    W_\vartheta(\delta_a, \mu)  \le \min \{W_\vartheta(\delta_{x}, \mu),W_\vartheta(\delta_{y}, \mu) \}. \end{equation}

    First, we observe that the right hand side in inequality \ref{ecuacionprop33} above is easy to compute. Using that $\varphi^{-1}$ is an increasing function,
    \begin{align*}
    W_\vartheta(\delta_x, \mu) &= \inf \left\{ t>0 : (1-\lambda) \varphi\left( \frac{1}{t} \dist(x, y) \right) \le 1 \right\} \\
    &=\inf \left\{ t>0 : \varphi\left( \frac{1}{t} \dist(x, y) \right) \le \frac{1}{1-\lambda} \right\} \\
    &=\inf \left\{ t>0 : \frac{1}{t}  \le \frac{\varphi^{-1} \left( \frac{1}{1-\lambda}\right)}{\dist(x,y)} \right\} 
    \\ &= \inf \left\{ t>0 : \frac{\dist(x,y)}{\varphi^{-1} \left( \frac{1}{1-\lambda}\right)} \le t \right\} \\
    &= \frac{\dist(x,y)}{\varphi^{-1} \left( \frac{1}{1-\lambda}\right)}.
    \end{align*}
Similarly, $W_\vartheta(\delta_y, \mu) =\dfrac{\dist(x,y)}{\varphi^{-1} \left( 1/\lambda\right)} $. If we want $\lambda$ close to one, we can suppose $\lambda > 1- \lambda$. Therefore, $1/{(1-\lambda)} > {1}/{\lambda}$, and because $\varphi^{-1}$ is increasing, 
\[
\varphi^{-1} \left(1/{(1-\lambda)} \right)> \varphi^{-1} \left({1}/{\lambda} \right).
\]
Thus, we know that
\[t_0 := \min \{W_\vartheta(\delta_{x}, \mu),W_\vartheta(\delta_{y}, \mu) \} = \frac{\dist(x,y)}{\varphi^{-1} \left( \frac{1}{1-\lambda}\right)}.\]
Now, we will show that we can find a point inside the geodesic $a = \gamma(s)$, $s\in (0, \ell)$ verifying inequality \eqref{ecuacionprop33}.
It suffices to see that
$t_0 \in S$, 
because $W_\vartheta(\delta_a, \mu)$ is the infimum of $S$ and by definition will be smaller. 
First, observe that, by monotonicity of $\varphi^{-1}$ the inequality defining $S$ is equivalent to
\[
\varphi^{-1} \left( \lambda \varphi\left( \frac1t \dist(a, x) \right) + (1-\lambda) \varphi \left( \frac{1}{t} \dist(a, y) \right) \right) \le \varphi^{-1} (1) = 1.
\]
By concavity of $\varphi^{-1}$, it is enough to have 
\[
\lambda \frac{1}{t} \dist(a, x) + (1-\lambda) \frac{1}{t} \dist(a, y) \le 1.
\]

We will evaluate $t=t_0$ and look for a condition in $s$ so the preceding inequality is verified. Observe that $\dist(a, x) = s$, $\dist(a, y) = \ell - s$ and $\dist(x,y) = \ell$. Then

\begin{align*}
    \lambda \frac{1}{t_0} \dist(a, x) + (1-\lambda) \frac{1}{t_0} \dist(a, y) \le 1 &\iff
    \lambda \frac{s}{\ell} \cdot \varphi^{-1} (1/(1-\lambda)) + (1-\lambda) \frac{\ell-s}{\ell} \cdot \varphi^{-1} (1/(1-\lambda))   \le 1 \\ &\iff \frac{\lambda s}{\ell} + \frac{\ell-s}{\ell} - \lambda \cdot \frac{\ell-s}{\ell} \le \frac{1}{\varphi^{-1} (1/(1-\lambda))} 
    \\ &\iff
    s \cdot ( 2\lambda -1) \le \ell \left( \frac{1}{\varphi^{-1} (1/(1-\lambda))} + 1-\lambda
    \right) \\
    &\iff s \le \ell \left( \frac{1}{\varphi^{-1} (1/(1-\lambda))} -(1-\lambda)
    \right) \cdot (2\lambda-1) ^{-1}.
\end{align*}
If we show that our bound for $s$ is strictly positive, the minimum will be attained inside the geodesic and we will finish the proof. Choosing $\lambda$ close enough to 1, we have $(2\lambda-1) >0$ and $1/(1-\lambda) > c$. Therefore, $\varphi^{-1}( 1/(1-\lambda)) - 1/(1-\lambda) < 0 $ and, because the function $t \mapsto 1/t$ is decreasing,  $\dfrac1{\varphi^{-1}( 1/(1-\lambda))} - (1-\lambda) > 0 $ and we have finished our proof.
\end{proof}
An immediate consequence from our proposition is the following theorem, providing us with examples of manifolds with zero reach inside their Orlicz--Wasserstein space:

\begin{thm}{\label{reachceroorlicz}}
Let $X$ be a geodesic metric space, and $x\in X$ a point such that there exists another $y\in X$ with the property that there exists at least two different minimising geodesics from $x$ to $y$. Suppose $X$ is isometrically embedded into an Orlicz-Wasserstein space $W_\vartheta(X)$. Then, for every $\varphi$ such that $\varphi(t_0) \neq t_0$ for some $t_0>1$, \[
\reach(x, X\subset W_\vartheta(X))=0.
\]    
In particular, if there exists a point $x\in X$ satisfying that property, $\reach(X\subset W_\vartheta(X))=0$ for every $p>1$. Also, in compact manifolds and non-simply connected manifolds, $\reach(x, X\subset W_\vartheta(X)) = 0$ for every $x\in X$.
\end{thm}
\begin{proof}
    The proof is identical to the one from Theorem \ref{teorema22}. It remains to see that $\varphi(t_0) \neq t_0$ implies the condition we ask for in Proposition \ref{proposicion33}. Indeed, the convexity and $\varphi(1) = 1$ imply $\varphi(t) > t$ for every $t>t_0$. And, because $\varphi^{-1}$ is increasing, we also have $t > \varphi^{-1}(t)$ for every $t>t_0$, which is what we need to apply Proposition \ref{proposicion33}. 
\end{proof}

\subsection{Positive Reach}
Similarly to the $p$--Wasserstein case, several results by Kell \cite{kell} imply that reflexive $\mathrm{CAT}(0)$--spaces inside some Orlicz--Wasserstein spaces have infinite reach.

\begin{thm}{\label{cat0orlicz}}
    Let $(X,\dist)$ be a reflexive $\mathrm{CAT}(0)$--space. Suppose $\varphi$ is a convex function which can be expressed as $\varphi(r) = \psi(r^p)$, where $\psi$ is another convex function and $p>1$. Then \begin{equation}{\label{reachpositivoorlicz}}
        \reach(X\subset W_{\vartheta}(X))=\infty,
    \end{equation}where $\psi\equiv\identity$ and $\varphi(1)=1$.
\end{thm}

\begin{proof}
    As we pointed in the proof of Corollary \ref{reachinfinitocat0}, $\mathrm{CAT}(0)$ spaces are strictly $p$-convex, so by \cite[Lemma A.2.]{kell} they are strictly Orlicz $\varphi$-convex. Thus, the result is derived directly from \cite[Theorem A.4.]{kell} which confirms the existence of unique barycenters for every $\mu\in W_{\vartheta}(X)$.
\end{proof}

\begin{rema}
    All proper metric spaces (i. e., those where every bounded closed set is compact) is reflexive \cite{huff, kell}. As Caprace pointed out in \cite{caprace}, \textit{symmetric spaces} of non--compact type (i.e. with non-positive sectional curvature and no non-trivial Euclidean factor) and \textit{Euclidean buildings} are proper $CAT(0)$ spaces and are examples for which Theorem \ref{cat0orlicz} holds.
\end{rema}

\section{Reach of the Persistence Diagram space}{\label{reachpersistence}}

In \cite[Theorem 19]{bubenik}, Bubenik and Wagner construct an explicit isometric embedding of bounded separable metric spaces into $(\dgm_\infty,w_\infty)$. \begin{align*}
    \varphi:(X,\dist)&\to(\dgm_\infty,w_\infty)\\
    x&\mapsto\{(2c(k-1),2ck+\dist(x,x_k))\}_{k=1}^\infty,
\end{align*}where $c>\diam(X)=\sup\{\dist(x,y)\colon x,y\in X\}$ and $\{x_k\}_{k=1}^\infty$ is a countable, dense subset of $(X,\dist)$. The authors stated that this embedding can be thought of as a shifted version of the Kuratowski embedding (for more information about this embedding see \cite{cuerno}).

\begin{thm}{\label{reachPD}}
   Let $(X,\dist)$ be a separable, bounded metric space and $(\dgm_{\infty},w_\infty)$ the space of persistence diagrams with the bottleneck distance. If $x\in X$ is an accumulation point, then \[
    \reach(x, X\subset \dgm_\infty)=0.\] 
    In particular, if $X$ is not discrete, $\reach(X\subset \dgm_\infty)=0.$
\end{thm}
\begin{proof}
    For every two points $x, y\in X$, we can construct a persistence diagram $P$ with at least those two points minimizing the bottleneck distance from the diagram $P$ to the embedded space $\varphi(X)$. That $P$ will be a midpoint between $\varphi(x)$ and $\varphi(y)$, so by choosing $y$ arbitrarily close to $x$, we will have a diagram with several barycenters ($x$ and $y$) that is also arbitrarily close to $x$. Therefore, $ \reach(x,X\subset \dgm_\infty)=0$ for every accumulation point $x\in X$, and, thus, $\reach(X\subset\dgm_\infty)=0$. 

    Then, it suffices to prove our first claim. For $x, y \in X$, choose the diagram
    \[P = \left\{\left(2c(k-1), 2ck + \frac{\dist(x, x_k) + \dist(y, x_k)}{2}\right)\right\}_{k=1}^\infty.\]
    Now, observe that
    \begin{align*}
    w_\infty(\varphi(x), P) &= 
    \sup_{k\in \mathbb{N}}\left| \dist(x, x_k) - \frac{\dist(y, x_k) + \dist(x, x_k)}{2} \right|
   \\ &=\sup_{k\in \mathbb{N}} \frac{|\dist(x, x_k) - \dist(y, x_k)|}{2} = \frac12 w_\infty(\varphi(x), \varphi(y)) = \frac {\dist(x,y)}{2}.
    \end{align*}
    And, by a symmetric argument,
    \[w_\infty(\varphi(y), P)=\frac {\dist(x,y)}{2}.\]
    Note that, similarly to the end of the proof of \cite[Theorem 19]{bubenik}, any other pairing between points of the diagrams would pair two points from different vertical lines. Those points would be at distance at least $2c$. On the other hand, any possibly unpaired points are at distance at least $c$ from the diagonal. So those pairings would have a cost bigger than $c>\dist(x,y)/2$, and therefore we always pair points in the same vertical lines.
    
    Now, if $z\in X$, we will see that $P$ is at distance at least $\frac12 \dist(x,y)$ from $z$. Indeed, we can give a lower bound for the distance simply by ommiting the supremum:    \begin{align*}
     w_\infty(\varphi(z), P) &= \sup_{k\in \mathbb{N}} \left|\dist(z, x_k) - \frac{\dist(x, x_k) + \dist(y, x_k)}{2} \right|\\&\ge \left|\dist(z, x_k) - \frac{\dist(x, x_k) + \dist(y, x_k)}{2} \right|.
        \end{align*}
    Looking at $x_k$ arbitrarily close to $z$, we get that
    \[
    w_\infty(\varphi(z), P) \ge \left| \frac{\dist(x, z) + \dist(y, z)}{2}\right| \ge \frac{\dist(x,y)}{2}.
    \]
    This proves that $P$ is not in the image of $\varphi$, and that $\varphi(x), \varphi(y)$ both minimize the distance from $P$ to $\varphi(X)$, as we wanted to see.
\end{proof}


    
\nocite{*}
\printbibliography
\end{document}